\documentclass{article}
\usepackage[utf8]{inputenc}
\usepackage{multirow}
\usepackage{amsmath}
\usepackage{amsfonts}
\usepackage{amssymb}
\usepackage{graphicx}
\usepackage{subcaption}
\usepackage[export]{adjustbox}
\usepackage{wrapfig}
\usepackage{amsthm}
\usepackage{mathrsfs}
\usepackage{biblatex} 
\usepackage{tikz}

\usepackage[margin=1.5in]{geometry}
 
\newtheorem{thm}{Theorem}

\newtheorem{cor}{Corollary}
\newtheorem*{cor*}{Corollary}
\newtheorem{lem}{Lemma}
\newtheorem{mydef}{Definition}

\newtheorem{prop}{Proposition}

\usepackage{quoting,xparse}

\NewDocumentEnvironment{pquotation}{m}
  {\begin{quoting}[
     indentfirst=true,
     leftmargin=\parindent,
     rightmargin=\parindent]\itshape}
  {\bywhom{#1}\end{quoting}}

\title{
{Family of Approximations for Dirichlet $L$-functions}\\
}
\author{Mohammed Alzergani}

\begin{document}

\maketitle

\maketitle
\begin{abstract}
    We introduce an infinite family of approximations for a Dirichlet $L$-function $L(s, \chi)$ arising from truncated Euler products. These approximations are entire functions and satisfy the same functional equation as $L(s, \chi)$. We provide numerical evidence of the accuracy of estimating values of $L(s, \chi)$ in the critical strip using these approximations. We then provide a precise expression for the error of the approximations and show that the error has exponential decay in largest prime considered in the truncated Euler product.
\end{abstract}

\section{Introduction}

The Riemann zeta function and Dirichlet $L$-functions are an important part of modern analytic number theory. There is a strong relationship between the zeros of the Riemann zeta function and Dirichlet $L$-functions and prime numbers (as is explored more in "Multiplicative Number Theory" by Davenport \cite{davenport}). There is much research regarding approximations and properties of the Riemann zeta function and Dirichlet $L$-functions \cite{matiy2} \cite{Gonek1} \cite{Gonek2} \cite{Gonek3}.

In 2017, Matiyasevich \cite{matiy1} constructed a family of approximations for the Riemann zeta function, and in 2022, M. Nastasescu and A. Zaharescu \cite{nastasescu} answered some questions that Matiyasevich raised in his 2017 paper. Notably, they provide a bound for the error of the approximations that Matiyasevich constructed. Other families of approximations for the Riemann zeta function were introduced by Gonek, Hughes, and Keating \cite{Gonek2} and Gonek \cite{Gonek1}, the first introducing a hybrid approximation and the second introducing one arising from a finite Euler product.

In his 2017 paper, Matiyasevich \cite{matiy1} states: 
\begin{quotation}{}
``Numerical examples show that this trick can work at least in the case of the Riemann’s zeta function (considered in this paper) as well as in the case of the \textit{Dirichlet L-functions} and \textit{Ramanujan L-function} (to be considered in subsequent papers), but the author has not established any rigorous estimates". 
\end{quotation}

That is precisely what this paper sets out to accomplish for the Dirichlet $L$-functions. We construct a similar family of approximations for Dirichlet $L$-functions to the one introduced by Matiyasevich for the Riemann zeta function, then we employ complex analytic methods to find and bound the error of our approximations.

In Section 1 we provide numerical data for our approximations' strength. In Section 2 we construct our family of approximations arising from truncated Euler products. In Section 3.1, we bound the principal part of the truncated Euler product, then in Section 3.2 we find an integral formula for the error, which in Section 3.3 we then express in terms of incomplete gamma functions. Finally, in Section 3.4 we show that our family converges uniformly and absolutely to the Dirichlet L-function by bounding the error in the limit when the number of primes considered approaches infinity. While we use similar complex analytic methods as used in \cite{nastasescu} for the Riemann zeta functions, we give a more precise expression for the error of the approximations. \\

Let $s_0$ be a fixed arbitrary complex number, $\chi$ be a primitive character modulo $q$, $\kappa(\chi)$ be $0$ if $\chi$ is an even character and $1$ if $\chi$ is odd, and $\epsilon(\chi)$ be the Gauss sum of $\chi$ divided by $i^\kappa \sqrt{q}$.

Furthermore let $L_u(s,\chi)$ be the truncated Euler product over primes less than or equal to $u$, and let $L_u^\approx(s,\chi)$ be our holomorphic approximation of $L(s,\chi)$ that satisfies the functional equation of $L(s, \chi)$.
Likewise let $\xi_u(s,\chi)$ be the completed version of $L_u(s,\chi)$, with $$\xi_u(s,\chi) = \left(\frac{q}{\pi}\right)^{\frac{s+\kappa(\chi)}{2}} \Gamma \left(\frac{s+\kappa(\chi)}{2}\right) L_u(s,\chi)$$ and let $\xi_u^\approx(s,\chi)$ be the completed version of $L_u^\approx(s,\chi)$.
 
We find the error for the completed $L$-function to be
\begin{equation*}
    \xi(s_0, \chi) - \xi^{\approx}_u(s_0,\chi) = \frac{1}{2 \pi i} \int_{Re(s) = \sigma} g(s,\chi) \left( \frac{L(s,\chi) - L_u(s,\chi)}{s - s_0} + \epsilon(\chi) \frac{ L(s,\overline{\chi}) - L_u(s,\overline{\chi})}{s - (1 - s_0)} \right) ds
\end{equation*}

and we find the following bound for the error

\begin{equation*}
    |\xi(s_0, \chi) - \xi^{\approx}_u(s_0,\chi)| \ll \begin{cases}
        \left(\frac{q}{\pi}\right) \Gamma \left(0, \frac{\pi u^2}{q}\right) & \text{if } \kappa(\chi) = 1\\
        \sqrt{\frac{q}{\pi}}\Gamma\left( -\frac{1}{2}, \frac{\pi u^2}{q} \right) & \text{if } \kappa(\chi) = 0.
    \end{cases}
\end{equation*} 
with $\Gamma$ being the incomplete Gamma function (see Definition \eqref{def} or p. 908, 910 of \textit{Table of Integrals, Series, and Products} \cite{integrals}).\\

We've also computed the values of the $L$-functions as well as the errors obtained by our approximations $L^\approx_u$ for three different primitive characters in Tables \ref{tab:1}, \ref{tab:2}, and \ref{tab:3}. This provides some numerical evidence that support the strength of our approximations along with the way the approximations quickly get stronger with additional prime factors. \\

For example, we find for $L(\frac{1}{2}+8i, \chi_{5,4})$ that the approximation $L^\approx_5(\frac{1}{2}+8i, \chi_{5,4})$ has an error on the order of $10^{-13}$ while $L^\approx_7(\frac{1}{2}+8i, \chi_{5,4})$ obtained by adding one more factor has an error on the order of $10^{-33}$.

\begin{table}[h!]
    \centering
    \begin{tabular}{| p{2cm}|p{2.2cm}|p{2.2cm}|p{2.2cm}|p{2.2cm}|p{2.2cm}  |}

    \hline
     \multicolumn{6}{|c|} {Table of values}\\
         \hline
         & $s_0=1/2 + 8i$&$s_0=1/2 + 9i$&$s_0=1/2 + 10i$&$s_0=1/2+11i$&$s_0=1/2+12i$\\
         \hline
         $L(s_0,\chi_{5,4}) $ & $1.59 + 0.20i$    &$0.83 - 0.90i$&   $0.07 + 0.26 i$ & $1.11 + 0.27 i$ & -$0.06 + 0.07i$\\ \hline 
         $|L(s_0,\chi_{5,4}) -  L^\approx_5(s_0,\chi_{5,4})|$&   $7.16 \cdot 10^{-13}$  & $1.61 \cdot 10^{-12}$   &$3.61 \cdot 10^{-12}$ & $8.07 \cdot 10^{-12}$ & $1.80 \cdot 10^{-11}$\\ \hline 
        $|L(s_0,\chi_{5,4}) -  L^\approx_7(s_0,\chi_{5,4})|$ &   $6.70 \cdot 10^{-33}$  & $1.51 \cdot 10^{-32}$   &$3.41 \cdot 10^{-32}$ & $7.64 \cdot 10^{-32}$ & $1.71 \cdot 10^{-31}$\\ \hline 
         $|L(s_0,\chi_{5,4}) -  L^\approx_{11}(s_0,\chi_{5,4})|$ &   $3.85 \cdot 10^{-46}$  & $8.68 \cdot 10^{-46}$   &$1.95 \cdot 10^{-45}$ & $4.39 \cdot 10^{-45}$ & $9.83 \cdot 10^{-45}$\\
         \hline
    \end{tabular}
    \caption{For the Dirichlet character represented as $\chi_5(4,\cdot)$ on the $L$-functions and modular forms database (LMFDB) \cite{lmfdb}. Note that $\chi_{5,4}$ is a primitive character with conductor $5$.}
    \label{tab:1}
\end{table}

\begin{table}[h!]
    \centering
    \begin{tabular}{| p{2cm}|p{2.2cm}|p{2.2cm}|p{2.2cm}|p{2.2cm}|p{2.2cm}  |}

    \hline
     \multicolumn{6}{|c|} {Table of values}\\
         \hline
         & $s_0=1/2 + 8i$&$s_0=1/2 + 9i$&$s_0=1/2 + 10i$&$s_0=1/2+11i$&$s_0=1/2+12i$\\
         \hline
         $L(s_0,\chi_{7,6}) $ & $1.16 + 2.51i$    &$4.19 + 0.06i$&   $1.70 - 2.72i$ & $-0.24-0.23 i$ & $0.50 - 0.30i$\\ \hline 
         $|L(s_0,\chi_{7,6}) -  L^\approx_5(s_0,\chi_{7,6})|$ &   $8.66 \cdot 10^{-23}$  & $1.84 \cdot 10^{-22}$   &$3.93 \cdot 10^{-22}$ & $8.40 \cdot 10^{-22}$ & $1.80 \cdot 10^{-21}$\\ \hline 
         $|L(s_0,\chi_{7,6}) -  L^\approx_7(s_0,\chi_{7,6})|$ &   $8.66 \cdot 10^{-23}$  & $1.84 \cdot 10^{-22}$   &$3.93 \cdot 10^{-22}$ & $8.40 \cdot 10^{-22}$ & $1.80 \cdot 10^{-21}$\\ \hline 
         $|L(s_0,\chi_{7,6}) -  L^\approx_{11}(s_0,\chi_{7,6})|$ &   $2.24 \cdot 10^{-32}$  & $6.90 \cdot 10^{-32}$   &$1.47 \cdot 10^{-31}$ & $3.15 \cdot 10^{-31}$ & $6.76 \cdot 10^{-31}$\\
         \hline
    \end{tabular}
    \caption{For the Dirichlet character represented as $\chi_7(6,\cdot)$ on the $L$-functions and modular forms database (LMFDB) \cite{lmfdb}. Note that $\chi_{7,6}$ is a primitive character with conductor $7$.}
    \label{tab:2}
\end{table}

\begin{table}[h!]
    \centering
    \begin{tabular}{| p{2cm}|p{2.2cm}|p{2.2cm}|p{2.2cm}|p{2.2cm}|p{2.2cm}  |}

    \hline
     \multicolumn{6}{|c|} {Table of values}\\
         \hline
         & $s_0=1/2 + 8i$&$s_0=1/2 + 9i$&$s_0=1/2 + 10i$&$s_0=1/2+11i$&$s_0=1/2+12i$\\
         \hline
         
         $L(s_0,\chi_{8,5}) $ & $0.16 + 0.89i$    &$2.61 + 0.53i$&   $0.93 - 1.62i$ & $0.26 + 0.27 i$ & $0.51 - 0.31i$\\ \hline 
         $|L(s_0,\chi_{8,5}) -  L^\approx_5(s_0,\chi_{8,5})|$ &   $1.02 \cdot 10^{-7}$  & $2.28 \cdot 10^{-7}$   &$5.08 \cdot 10^{-7}$ & $1.13 \cdot 10^{-6}$ & $2.50 \cdot 10^{-6}$\\ \hline 
         $|L(s_0,\chi_{8,5}) -  L^\approx_7(s_0,\chi_{8,5})|$ &   $2.28 \cdot 10^{-20}$  & $5.13 \cdot 10^{-20}$   &$1.15 \cdot 10^{-19}$ & $2.59 \cdot 10^{-19}$ & $5.78 \cdot 10^{-19}$\\ \hline
         $|L(s_0,\chi_{8,5}) -  L^\approx_{11}(s_0,\chi_{8,5})|$ &   $1.07 \cdot 10^{-28}$  & $2.41 \cdot 10^{-28}$   &$5.43 \cdot 10^{-28}$ & $1.22 \cdot 10^{-27}$ & $2.73 \cdot 10^{-27}$\\
         \hline
    \end{tabular}
    \caption{For the Dirichlet character represented as $\chi_8(5,\cdot)$ on the $L$-functions and modular forms database (LMFDB) \cite{lmfdb}. Note that all $\chi_{8,5}$ is a primitive character with conductor $8$.}
    \label{tab:3}
\end{table}



\newpage
\section{A Family of Approximations for Dirichlet L-functions}


Throughout the paper we will let $\chi$ be a fixed primitive Dirichlet character of conductor $q$. Consider 
\begin{equation*}
  g(s,\chi) = \left(\frac{q}{\pi}\right)^{\frac{s+\kappa(\chi)}{2}} \Gamma \left(\frac{s+\kappa(\chi)}{2}\right)
\end{equation*}
with 
\begin{equation*}
\kappa(\chi) = \begin{cases}
    0 & \chi(-1) = 1 \\
    1 & \chi(-1) = -1.
\end{cases}
\end{equation*}

We define a truncated Euler product $L_u (s, \chi)$ of $L(s, \chi)$ as
\begin{equation*}
    L_u(s,\chi) = \prod_{\text{prime } p \le u} \frac{1}{1-\frac{\chi(p)}{p^s}},
\end{equation*} 
where the product is over primes less than or equal to $u$. We also introduce the completed analog $\xi_u(s, \chi)$ of $L_u (s, \chi)$ as
\begin{equation*}
 \xi_u(s,\chi) = g(s,\chi) L_u(s,\chi).    
\end{equation*}

For future convenience, we introduce the modified truncated Euler product $L_{u}^{\setminus\{p_j\}} (s, \chi)$ with a single prime $p_j \leq u$ removed
\begin{align*}
    L_{u}^{\setminus\{p_j\}} (s, \chi) =  \prod_{\substack{ p \le u \\ p \neq p_j}} \frac{1}{1-\frac{\chi(p)}{p^s}}, 
\end{align*}
as well as its completed analog
\begin{equation*}
\xi_{u}^{\setminus\{p_j\}} = g(s,\chi)  L_{u}^{\setminus\{p_j\}} (s, \chi).    
\end{equation*}

Our goal is to introduce a family of approximations $\xi_u^\approx (s, \chi)$ of $\xi(s, \chi)$ arising from the truncated Euler product $\xi_u (s, \chi)$ introduced above such that each $\xi_u ^\approx (s, \chi)$ is an entire function and satisfies the same functional equation 
\begin{equation*}
\xi_u^\approx(s,\chi) = \epsilon(\chi)\xi_u^\approx(1-s,\overline{\chi})
\end{equation*}
as $\xi(s, \chi)$. In order to do this, we need to first find the principal part $\xi_u^{pp} (s, \chi)$ of $\xi_u (s, \chi)$. 

\begin{lem}
    Given $\xi_u (s,\chi)$ as above, its principal part is given by
\begin{align*}
    \xi^{pp}_{u}(s,\chi) &= \sum_{\substack{j=1 \\ p_j \nmid q}}^{m} \ 
    \sum_{\substack{h\in\mathbb{Z} \\ h \neq \frac{-n_j}{\varphi (q)}}} 
    \frac{\xi_{u}^{\setminus\{p_j\}}\left(\frac{2\pi i}{\log (p_j)}\left(\frac{n_j}{\varphi (q)} + h\right),\chi\right)}{(s-\frac{2\pi i}{\log (p_j)}(\frac{n_j}{\varphi (q)}+h))\log(p_j)} \\
    &+\sum_{h\in\mathbb{Z}_+}\frac{2(-1)^h\pi^h L_u(-2h-\kappa(\chi),\chi)}{q^h\Gamma(h+1)(s+2h+\kappa(\chi))} \\
    &+ \text{principal part from s=0}.
\end{align*}
\end{lem}

\begin{proof}
    With the possible exception of the pole at $s=0$, every pole of $\xi_u(s, \chi)$ is a simple one, and so the principal part is the sum 
\begin{equation*}
\sum_{s_0} \frac{Res(s_0)}{s-s_0} 
\end{equation*}
over all the poles $s_0$ of $\xi_u(s,\chi)$ plus a term from the pole at $s=0$.

The function $\Gamma(n)$ has simple poles at each $n \in \mathbb{Z}_{\le 0}$ with residue 
\begin{equation*}
\frac{(-1)^{-n}}{\Gamma(-n+1)}.
\end{equation*}
Therefore when $s=-\kappa-2h$ with $h \in \mathbb{Z}_{\ge 0}$, the residue of $\xi_u(s,\chi)$ is 
\begin{equation*}
    \frac{2(-1)^h}{\Gamma(h+1)}\left(\frac{q}{\pi}\right)^{-h} L_u(-2h-\kappa(\chi),\chi).
\end{equation*}
$L_u(s,\chi)$ has simple poles when a term in the Euler product has denominator 0. This occurs when, for the $j$-th prime $p_j<u$,
\begin{equation*}
    \frac{\chi(p_j)}{{p_j}^s}=1.
\end{equation*}
Note that there is no pole contributed by the j-th term whenever $p_j | q$.
Since $\chi(p_j)$ is a $\varphi(q)$-th root of unity, there is an integer $n_j$ between 0 and $\varphi(q)$ that satisfies 
\begin{equation*}
    \chi(p_j) = \exp{\left(2 \pi i \left( \frac{n_j}{\varphi(q)} + h \right)\right)}
\end{equation*}
for every $h \in \mathbb{Z}$. Hence the poles $s$ of $L_u(s, \chi)$ are at points
\begin{equation*}
    s = \frac{2 \pi i}{\log(p_j)} \left( \frac{n_j}{\varphi(q)} + h \right).
\end{equation*}
The residue of $\xi_u$ at these poles is therefore
\begin{equation*}
    \frac{\xi_{u}^{\setminus\{p_j\}}\left(\frac{2\pi i}{\log (p_j)}\left(\frac{n_j}{\varphi (q)} + h\right),\chi\right)}{\log(p_j)}.
\end{equation*}

\end{proof}

To complete our approximation, we get the holomorphic part of the Laurent expansion 
\begin{align*}
    \xi^{reg}_{u}(s,\chi) &= \xi_{u}(s,\chi) - \xi^{pp}_{u}(s,\chi),
\end{align*}
and take into account the functional equation to give us our approximation $\xi_u^\approx$ defined by
\begin{equation*}
    \xi^{\approx}_{u}(s,\chi) = \xi^{reg}_{u}(s,\chi) + \epsilon(\chi) \xi^{reg}_{u}(1-s,\overline{\chi})
\end{equation*}
where
\begin{equation*}
    \epsilon(\chi) = \frac{\tau(\chi)}{i^\kappa \sqrt{q}} = \frac{\sum_{n=1}^q \chi(n)e^{2\pi i n/q}}{i^\kappa \sqrt{q}}.
\end{equation*}

\begin{lem}
    $\xi_u^\approx (s, \chi)$ satisfies the same functional equation of $\xi_u(s,\chi)$. 
\end{lem}
\begin{proof}
    Note that for primitive character modulo $q$,
    \begin{align*}
        \epsilon(\chi) \epsilon(\overline{\chi}) &= \frac{\tau(\chi)\overline{\tau(\chi)}\chi(-1)}{i^{\kappa(\chi)}i^{\kappa(\overline{\chi})}\sqrt{q}\sqrt{q}}\\ &= \frac{\lvert \tau(\chi)\rvert ^ 2 \chi(-1)}{i^{2\kappa(\chi)}q} = \frac{\chi(-1)}{(-1)^{\kappa(\chi)}}\\ &= 1. 
    \end{align*}
    Therefore $\xi_u^\approx$ indeed satisfies the functional equation
\begin{align*}
    \epsilon(\chi)\xi_u^\approx(1-s,\overline{\chi})
    &= \epsilon(\chi)\xi^{reg}_{u}(1-s,\overline{\chi})+\epsilon(\chi)\epsilon(\overline{\chi})\xi^{reg}_{u}(s,\chi) \\
    &= \epsilon(\chi)\xi^{reg}_{u}(1-s,\overline{\chi})+\xi^{reg}_{u}(s,\chi) \\
    &= \xi_u^\approx(s,\chi).
\end{align*}

\end{proof}

We finish by defining our approximation for $L(s,\chi)$:
\begin{equation}
    L_u^\approx (s, \chi) = \frac{\xi_u^\approx(s,\chi)}{g(s,\chi)}. \label{eq:1}
\end{equation}



\section{Main Results}
The main theorem and its corollary we prove in this paper are the following:
\begin{thm}
\label{thm1}
    Let $\chi$ be a fixed primitive Dirichlet character of conductor $q$. Fix $s_0 \in \mathbb{C}$ and let $A_u$ be the set of all positive integers with at least one prime factor greater than u. Then
    \begin{equation}
    \xi(s_0, \chi) - \xi^{\approx}_u(s_0,\chi) = \sum_{n \in A_u} J(n)
    \end{equation}
where
\begin{align}
    J(n) &= \chi(n) n^{\kappa(\chi)} \left( \frac{q}{n^2 \pi} \right)^{\frac{\kappa(\chi) + s_0}{2}}  \Gamma\left( \frac{\kappa(\chi) + s_0}{2}, \frac{n^2 \pi}{q} \right) \\
    &+ \epsilon(\chi) \overline{\chi}(n) n^{\kappa(\chi)} \left( \frac{q}{n^2 \pi} \right)^{\frac{\kappa(\chi) + 1 - s_0}{2}}  \Gamma\left( \frac{\kappa(\chi) + 1 - s_0}{2}, \frac{n^2 \pi}{q} \right).
\end{align}
and where
\begin{equation*}
        \Gamma(z,x) = \int_x^\infty e^{-t} t^{z-1} dt.
    \end{equation*}
\end{thm}

\begin{cor}
\label{cor}
Let $\chi$ be a fixed primitive Dirichlet character of conductor $q$ and fix $s_0 \in \mathbb{C}$. Then as $u \rightarrow \infty$,
\begin{equation*}
    |\xi(s_0, \chi) - \xi^{\approx}_u(s_0,\chi)| \ll \begin{cases}
        \left(\frac{q}{\pi}\right) \Gamma \left(0, \frac{\pi u^2}{q}\right) & \text{if } \kappa(\chi) = 1\\
        \sqrt{\frac{q}{\pi}}\Gamma\left( -\frac{1}{2}, \frac{\pi u^2}{q} \right) & \text{if } \kappa(\chi) = 0.
    \end{cases}
\end{equation*}
\end{cor}

We also have the following intermediary theorem we prove in Section 3.2:\\
\textbf{Theorem 2.} \textit{Fix $\sigma>1$. We have}
\begin{align*}
   \xi(s_0, \chi) - \xi^{\approx}_u(s_0,\chi) = \frac{1}{2 \pi i} \int_{Re(s) = \sigma} g(s,\chi) \left( \frac{L(s,\chi) - L_u(s,\chi)}{s - s_0} + \epsilon(\chi) \frac{ L(s,\overline{\chi}) - L_u(s,\overline{\chi})}{s - (1 - s_0)} \right) ds.
    \end{align*}

\subsection{Bounding the principal part}

In this section we will will bound the principal part. However we will need to use the Stirling approximation:
\begin{lem}
\label{lem:4}
\begin{itemize}
\item[i.]  $\Gamma(z) \sim e^{-z} z^{z-\frac{1}{2}} \sqrt{2 \pi} \left(1 + \frac{1}{12z} + O\left( \frac{1}{z^2} \right)\right)$
\item [ii.]  $\left| \Gamma(z) \right| \sim \sqrt{2 \pi} e^{- \mathfrak{Re}(z)} \left| z \right| ^{\mathfrak{Re}(z) - \frac{1}{2}} e^{- \mathfrak{Im}(z) \arg(z)}$
\item [iii.] If  $\mathfrak{Im}(z) \rightarrow \infty$, then $\left| \Gamma(z) \right| \sim c \left| z \right| ^{\mathfrak{Re}(z) - \frac{1}{2}} e^{-\left| z \right| \frac{\pi}{2}}$
\item [iv.] Let $x \ge 1, x,y\in\mathbb{R}$, then $\left| \frac{\Gamma(x+iy)}{\Gamma(x)} \right| \ll \exp \{ - \frac{3}{8} \min \left( \left| y \right|, \frac{y^2}{x} \right) \}.$
\end{itemize}
\end{lem}
\begin{proof}
    These are obtained from Nastasescu and Zaharescu's paper \cite{nastasescu}.
\end{proof}

The bound for the principal part that we will prove in this section is the following:
\begin{lem}
    \begin{equation}
        \left| \xi^{pp}_{u}(s,\chi)  \right| \le \frac{C'_u}{1 + \left| s \right|}
         \label{lem:7}
    \end{equation}
\end{lem}

However first 

\begin{proof}
Fix $u \in \mathbb{Z}$ an integer $\ge 2$.

By the Prime Number Theorem, we can select $s$ such that
\begin{equation}
    \min_{1 \le j \le m} \min_{h \in \mathbb{Z}} \left| s - \frac{2 \pi i}{\log (p_j)} \left( \frac{n_j}{\varphi(q)} + h \right) \right| \gg \frac{1}{u} \label{eq:5}
\end{equation}
and
\begin{equation}
    \min_{h \in \mathbb{Z}_+} \left| s + 2h + \kappa(\chi) \right| \gg 1. \label{eq:6}
\end{equation}
This combined with $(iii)$ in Lemma \eqref{lem:4} and $(iv)$ in Lemma \eqref{lem:4} implies that

\begin{equation}
    \max_{1 \le j \le m} \max_{L \le \left| \frac{2 \pi}{\log (p_j)} \left( \frac{n_j}{\varphi(q)} + h \right) \right| \le L+1} \left| \frac{\xi_{u}^{\setminus\{p_j\}}\left(\frac{2\pi i}{\log (p_j)}\left(\frac{n_j}{\varphi (q)} + h\right),\chi\right)}{(s-\frac{2\pi i}{\log (p_j)}(\frac{n_j}{\varphi (q)}+h))\log(p_j)}
    \right| \le C_u e^{- C_0 L}.
\end{equation}
Then using notation $s_{j,h} \in \mathbb{C}$ for the pole of the form $\frac{2\pi i}{\log (p_j)}\left(\frac{n_j}{\varphi (q)} + h\right)$ for convenience, the part of the principal part arising from the poles $s_{j,h}$ is bounded by:
\begin{align*}
    \sum_{L=0}^{L=\frac{|s|}{2}} \sum_{L \le |s_{j,h}| \le L+1} \left| \frac{C_ue^{-C_0L}}{s-s_{j,h}} \right| &+ \sum_{L=\frac{|s|}{2}}^{L=2|s|} \sum_{L \le |s_{j,h}| \le L+1} \left| \frac{C_ue^{-C_0L}}{s-s_{j,h}} \right|\\ &+ \sum_{L=2|s|}^\infty \sum_{L \le |s_{j,h}| \le L+1} \left| \frac{C_ue^{-C_0L}}{s-s_{j,h}} \right|.
\end{align*}
For $L < \frac{|s|}{2}$, the denominator $|s-s_{j,h}|$ is greater than $c_{u,1}(|s|+1)$ for an appropriate constant $c_{u,1}$, hence
\begin{equation*}
\sum_{L=0}^{\frac{|s|}{2}} \sum_{L \le |s_{j,h}| \le L+1} \left| \frac{C_ue^{-C_0L}}{s-s_{j,h}} \right| \leq \frac{C_{u,1}}{1+|s|}.
\end{equation*}
For $\frac{|s|}{2} < L < 2|s|$, we have exponential decay when $|s|$ is large, and the sparseness condition in \eqref{eq:5} on $s$ guarantees that $\frac{1}{|s-s_{j,h}|}$ cannot get too large. Hence
\begin{equation*}
\sum_{L=\frac{|s|}{2}}^{L=2|s|} \sum_{L \le |s_{j,h}| \le L+1} \left| \frac{C_ue^{-C_0L}}{s-s_{j,h}} \right| \leq \frac{C_{u,2}}{1+|s|}.
\end{equation*}

For $L > 2|s|$, the exponential decay guarantees that 
\begin{equation*}
\sum_{L=2|s|}^\infty \sum_{L \le |s_{j,h}| \le L+1} \left| \frac{C_ue^{-C_0L}}{s-s_{j,h}} \right| \leq \frac{C_{u,3}}{1+|s|}.
\end{equation*}

Since the $\Gamma$-function rapidly decays, 
\begin{equation}
    \left|\sum_{h\in\mathbb{Z}_+}\frac{2(-1)^h\pi^h L_u(-2h-\kappa(\chi),\chi)}{q^h\Gamma(h+1)(s+2h+\kappa(\chi))} \right| \leq \frac{C_{\gamma}}{1 + |s|}.
\end{equation}
If the pole of $\xi_u$ at $s=0$ is of order $\ell$, the principal part at $s=0$ is of the form
\begin{equation*}
    \text{principal part from s=0} = \sum_{i=1}^{\ell} \frac{c_i}{s^{i}}
\end{equation*}
which is finite with each term asymptotically less than $\frac{1}{1+|s|}$, hence
\begin{equation*}
    \text{principal part from s=0} \le \frac{C_{u,0}}{1+|s|}.
\end{equation*}
Therefore there exists a constant $C'_u = C_{u,0} + C_{u,1} + C_{u,2} + C_{u,3} + C_{\gamma}$ such that
\begin{equation}
    \left| \xi_u^{pp}(s) \right| \le \frac{C'_u}{1+|s|}.
\end{equation}
\end{proof}

\subsection{Integral formula for the error}

Fix a complex number $s_0$ throughout the section. We will prove an integral formula for the error of the approximation
$$\xi (s_0, \chi) - \xi_u^\approx(s_0, \chi).$$
More specifically,  we will show the following exact expression:
\begin{thm}
\label{thm2}
Fix $\sigma>1$. We have
\begin{align*}
   \xi(s_0, \chi) - \xi^{\approx}_u(s_0,\chi) = \frac{1}{2 \pi i} \int_{Re(s) = \sigma} g(s,\chi) \left( \frac{L(s,\chi) - L_u(s,\chi)}{s - s_0} + \epsilon(\chi) \frac{ L(s,\overline{\chi}) - L_u(s,\overline{\chi})}{s - (1 - s_0)} \right) ds.
    \end{align*}
\end{thm}

\begin{proof}
Choose a simple closed curve $\mathcal{C}$ encompassing points $s_0$ and $1-s_0$ that doesn't pass through any poles of $\xi_u(s, \chi)$ on the real or imaginary axis. Then we define the following:

\begin{equation*}
    \mathcal{I}(u,s_0,\mathcal{C}) := \frac{1}{2 \pi i} \int_\mathcal{C} \frac{\xi^{pp}_u(s,\chi)}{s - s_0} + \frac{\xi^{pp}_u(s,\overline{\chi})}{s - (1 - s_0)} ds.
\end{equation*}

Next we construct a sequence of rectangular closed curves $\mathcal{C}_l$ for $l \in \mathbb{Z_+}$ each encompassing $\mathcal{C}$ while satisfying \eqref{eq:5} and \eqref{eq:6} and $\mathcal{C}_l \subseteq \mathcal{C}_{l+1}$ and $\bigcup \text{Int}(\mathcal{C}_l) = \mathbb{C}$. 

For a fixed $l \in \mathbb{Z_+}$, by Cauchy's Residue formula, if we let $\rho$ denote the poles on the real and complex axes lying in between $\mathcal{C}$ and $\mathcal{C}_l$ and Res($\rho$) denote the residue of $$\frac{\xi^{pp}_u(s,\chi)}{s - s_0} + \frac{\xi^{pp}_u(s,\overline{\chi})}{s - (1 - s_0)}$$ at each $\rho$, then

\begin{equation}
    \mathcal{I}(u,s_0,\mathcal{C}) := \frac{1}{2 \pi i} \int_{\mathcal{C}_l} \frac{\xi^{pp}_u(s,\chi)}{s - s_0} + \frac{\xi^{pp}_u(s,\overline{\chi})}{s - (1 - s_0)}ds - \sum_\rho \text{Res}(\rho).
\end{equation}

But $\mathcal{C}_l$ satisfying \eqref{eq:5} \& \eqref{eq:6} combined with \eqref{lem:7} means that $$\frac{\xi^{pp}_u(s,\chi)}{s - s_0} + \frac{\xi^{pp}_u(s,\overline{\chi})}{s - (1 - s_0)} \ll \frac{1}{\left| s \right| ^2}$$ which tells us that as $l \rightarrow \infty$, $\int_{\mathcal{C}_l} \frac{\xi^{pp}_u(s,\chi)}{s - s_0} + \frac{\xi^{pp}_u(s,\overline{\chi})}{s - (1 - s_0)} ds \rightarrow 0$. So as $l \rightarrow \infty$, 
\begin{equation*}
    \mathcal{I}(u,s_0,\mathcal{C}) = - \sum_\rho \text{Res}(\rho)
\end{equation*}
with $\rho$ denoting the poles on the real and complex axis lying outside of $\mathcal{C}$. Note that here 
\begin{equation*}
\sum_\rho \text{Res}(\rho) = \lim_{T \rightarrow \infty} \sum_{\left|\rho\right| \le T} \text{Res}(\rho).
\end{equation*}

For $T, \lambda, \sigma \in \mathbb{R}$, let $\mathcal{C}$ and $\mathcal{C'}$ denote the counterclockwise rectangular contours with vertices $\sigma - i T, \sigma + i T, -\lambda + i T, -\lambda - i T$ and $1 + \lambda - i T, 1 + \lambda + i T, -\lambda + i T, -\lambda - i T$ respectively. Choose $T>1$, $\lambda>1$, and $\sigma>1$ such that $s_0$ is enclosed inside both $\mathcal{C}$ and $\mathcal{C'}$.

By the Cauchy integral formula,
\begin{equation}
    \xi^{reg}_{u}(s_0,\chi) = \frac{1}{2 \pi i} \int_{\mathcal{C}'} \frac{\xi^{reg}_{u}(s,\chi)}{s-s_0}ds,
\end{equation}
therefore
\begin{align*}
    \xi^{\approx}_{u}(s,\chi) = \frac{1}{2 \pi i} \int_{\mathcal{C}'} \frac{\xi_u(s,\chi)}{s - s_0} + \frac{\xi_u(s,\overline{\chi})}{s - (1 - s_0)}ds - \frac{1}{2 \pi i} \int_{\mathcal{C}'} \frac{\xi^{pp}_u(s,\chi)}{s - s_0} + \frac{\xi^{pp}_u(s,\overline{\chi})}{s - (1 - s_0)}ds.
\end{align*}
Therefore, if we define $E(u, s_0, T, \lambda)$ to be 
\begin{equation*}
E(u, s_0, T, \lambda) =- \frac{1}{2 \pi i} \int_{\mathcal{C}'} \frac{\xi^{pp}_u(s,\chi)}{s - s_0} + \frac{\xi^{pp}_u(s,\overline{\chi})}{s - (1 - s_0)}ds = -\mathcal{I}(u, s_0, \mathcal{C}'),
\end{equation*}
we have
\begin{align*}
   \xi^{\approx}_{u}(s,\chi) &= \frac{1}{2 \pi i} \int_{\mathcal{C}'} \frac{\xi_u(s,\chi)}{s - s_0} + \frac{\xi_u(s,\overline{\chi})}{s - (1 - s_0)}ds + E(u, s_0, T, \lambda).\\
\end{align*}
As before, if $T, \lambda \rightarrow \infty$, then  $$\mathcal{I}(u,s_0,\mathcal{C}') = - \sum_\rho \text{Res}(\rho),$$ with $\rho$ denoting the poles on the real and complex axis lying outside of $\mathcal{C}'$, ie. poles of the form $-2h -\kappa(\chi)$ with $2h + \kappa(\chi) > \lambda$ or of the form $\frac{2\pi i}{\log (p_j)}\left(\frac{n_j}{\varphi (q)} + h\right)$ with $\frac{2\pi}{\log (p_j)}\left|\frac{n_j}{\varphi (q)} + h\right| > T$. 

Since $\text{Res}(\rho) \ll \frac{1}{\left| s \right| ^2}$ for every $\rho$ by Lemma \eqref{lem:7}, this means that as $T, \lambda \rightarrow \infty$, $E(u, s_0, T, \lambda)\rightarrow 0$. \\

To get the expression for the error of the approximation, we use the functional equation $\xi(s_0, \chi) = \epsilon(\chi) \xi(1 - s_0, \overline{\chi})$ which gives 
\begin{equation*}
\xi(s_0, \chi) = \frac{1}{2}(\xi(s_0, \chi) + \epsilon(\chi) \xi(1 - s_0, \overline{\chi})). 
\end{equation*}
Using Cauchy's Residue Theorem, we find:

\begin{equation*}
    \xi(s_0, \chi) = \frac{1}{2 \pi i} \int_{\mathcal{C}'} \frac{1}{2} \left(\frac{\xi(s,\chi)}{s - s_0} \right) ds + \frac{1}{2 \pi i} \int_{\mathcal{C}'} \frac{1}{2} \left(\frac{\epsilon(\chi) \xi(s,\overline{\chi})}{s - (1 - s_0)} \right) ds.
\end{equation*}
Let us define the function:
\begin{align*}
    I(u, s_0, T, \lambda) &= \frac{1}{2 \pi i} \int_{\mathcal{C}'} \frac{1}{2} \left(\frac{\xi(s,\chi)}{s - s_0} \right) ds + \frac{1}{2 \pi i} \int_{\mathcal{C}'} \frac{1}{2} \left(\frac{\epsilon(\chi) \xi(s,\overline{\chi})}{s - (1 - s_0)} \right) ds \\
    &- \frac{1}{2 \pi i} \int_{\mathcal{C}'} \frac{\xi_u(s,\chi)}{s - s_0} ds - \frac{1}{2 \pi i} \int_{\mathcal{C}'} \frac{\epsilon(\chi) \xi_u(s,\overline{\chi})}{s - (1 - s_0)} ds.
\end{align*}

    Hence we can simplify to find $I(u, s_0, T, \lambda) - E(u, s_0, T, \lambda)$:
    \begin{align*}
    &= \xi(s_0, \chi) - \frac{1}{2 \pi i} \int_{\mathcal{C}'} \frac{\xi_u(s,\chi) - \xi^{pp}_u(s,\chi)}{s - s_0}  ds - \frac{1}{2 \pi i} \int_{\mathcal{C}'} \frac{\epsilon(\chi) (\xi_u(s,\overline{\chi}) - \xi^{pp}_u(s, \overline{\chi}))}{s - (1 - s_0)}  ds \\
    &= \xi(s_0, \chi) - (\xi^{reg}_u(s_0,\chi) + \epsilon(\chi) \xi^{reg}_u(1 - s_0, \overline{\chi})) \\
    &= \xi(s_0, \chi) - \xi^{\approx}_u(s_0,\chi).\\
\end{align*}

\begin{figure}[htp]
    \centering
    \begin{tikzpicture}
        \fill [orange!25] (4,4.7) rectangle (5,5);
        \fill [orange!25] (4,1.3) rectangle (5,1);
        \fill [orange!25] (4.7,0.7) rectangle (5.3,5.3);
    
        \fill [green!25] (4,5) rectangle (7,5.3);
        \fill [green!25] (4,1) rectangle (7,0.7);
        \fill [green!25] (6.7,0.7) rectangle (7.3,5.3);
    
        \fill [blue!25] (1,4.7) rectangle (4,5.3);
        \fill [blue!25] (1,0.7) rectangle (4,1.3);
        \fill [blue!25] (0.7,0.7) rectangle (1.3,5.3);

        \draw (1,1) rectangle (7,5);
        \draw [ultra thick, ->] (0,3) -- (8,3);
        \draw [ultra thick, ->] (3,0) -- (3,6);
        \draw (4,6) -- (4,0);
        \draw (5,5) -- (5,1);
        \filldraw (1,3) circle (2pt) node[anchor = north east] {$- \lambda$};
        \filldraw (3,3) circle (2pt) node[anchor = north east] {$O$};
        \filldraw (4,3) circle (2pt) node[anchor = north east] {$\frac{1}{2}$};
        \filldraw (5,3) circle (2pt) node[anchor = north west] {$\sigma$};
        \filldraw (7,3) circle (2pt) node[anchor = north west] {$\lambda + 1$};
        \filldraw (3,1) circle (2pt) node[anchor = north east] {$-iT$};
        \filldraw (3,5) circle (2pt) node[anchor = south east] {$iT$};
        
        \foreach \x in {1.25,1.5,1.75,2,2.25,2.5,2.75}
            \draw [thick, red] (\x cm,3.1) -- (\x cm,2.9);
        \foreach \y in {1.25,1.5,1.75,2,2.25,2.5,2.75,3.25,3.5,3.75,4,4.25,4.5,4.75}
            \draw [thick, red] (2.9,\y cm) -- (3.1,\y cm);

        \fill [orange!25] (8.75,3.8) rectangle (9.75,4.2);
        \fill [green!25] (8.75,3.3) rectangle (9.75,3.7);
        \fill [blue!25] (8.75,2.8) rectangle (10.75,3.2);

        \filldraw [red] (9,2) circle (2pt) node[anchor = west] { $  $ poles in $\mathcal{C}$ and $\mathcal{C}'$};
        \filldraw (9,4) circle (2pt) node[anchor = west] {$  $ $\mathcal{C}_{\mathscr{R}}$};
        \filldraw (9,3.5) circle (2pt) node[anchor = west] {$  $ $\mathcal{C}'_{\mathscr{R}}$};
        \filldraw (9,3) circle (2pt) node[anchor = west] {$  $ $\mathcal{C}'_{\mathscr{L}} =\mathcal{C}_{\mathscr{L}}$};
        
    \end{tikzpicture}
    \caption{The four parts of $\mathcal{C}$ and $\mathcal{C}'$. Not to scale.}
    \label{fig:1}
\end{figure}
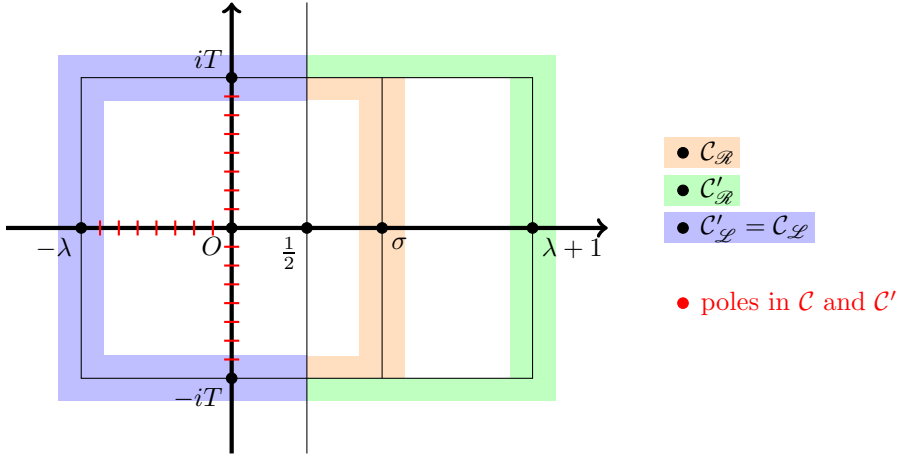

Now, let us split $\mathcal{C}$ and $\mathcal{C}'$ into left and right parts, divided at the critical line as shown in Figure \ref{fig:1}. Define $\mathcal{C}_{\mathscr{L}}$, $\mathcal{C}_{\mathscr{R}}$, $\mathcal{C}'_{\mathscr{L}}$, and $\mathcal{C}'_{\mathscr{R}}$ respectively. Note that  $\mathcal{C}'_{\mathscr{L}} =\mathcal{C}_{\mathscr{L}}$. Also note that for each point $z$ on $\mathcal{C}'_{\mathscr{R}}$, $1-z$ is on $\mathcal{C}'_{\mathscr{L}}$ and vice versa so by the functional equation $$\frac{1}{2 \pi i} \int_{\mathcal{C}'_{\mathscr{L}}} \frac{\xi(s,\chi)}{s - s_0} ds = \frac{1}{2 \pi i} \int_{\mathcal{C}'_{\mathscr{R}}} \frac{\epsilon(\chi)\xi(s,\overline{\chi})}{s - (1 - s_0)} ds$$ and $$\frac{1}{2 \pi i} \int_{\mathcal{C}'_{\mathscr{R}}} \frac{\xi(s,\chi)}{s - s_0} ds = \frac{1}{2 \pi i} \int_{\mathcal{C}'_{\mathscr{L}}} \frac{\epsilon(\chi)\xi(s,\overline{\chi})}{s - (1 - s_0)} ds.$$ This however does not apply to $\xi_u$ but since there are no poles of $\xi_u$ between $\mathcal{C}_{\mathscr{R}}$ and $\mathcal{C}'_{\mathscr{R}}$ by Cauchy Residue Theorem, 
$$\frac{1}{2 \pi i} \int_{\mathcal{C}'_{\mathscr{R}}} \frac{\xi_u(s,\chi)}{s - s_0} ds = \frac{1}{2 \pi i} \int_{\mathcal{C}_{\mathscr{R}}} \frac{\xi_u(s,\chi)}{s - s_0} ds$$ and $$\frac{1}{2 \pi i} \int_{\mathcal{C}'_{\mathscr{R}}} \frac{\epsilon(\chi)\xi(s,\overline{\chi})}{s - (1 - s_0)} ds = \frac{1}{2 \pi i} \int_{\mathcal{C}_{\mathscr{R}}} \frac{\epsilon(\chi)\xi(s,\overline{\chi})}{s - (1 - s_0)} ds.$$ Let us rewrite $I$ into two contours $$I(u, s_0, T, \lambda) = I_{\mathscr{L}}(u, s_0, T, \lambda) + I_{\mathscr{R}}(u, s_0, T, \lambda)$$ with
\begin{align*}
    I_{\mathscr{L}}(u, s_0, T, \lambda) &=- \frac{1}{2 \pi i} \int_{\mathcal{C}_{\mathscr{L}}} g(s,\chi)\frac{L_u(s,\chi)}{s - s_0} ds - \frac{1}{2 \pi i} \int_{\mathcal{C}_{\mathscr{L}}} g(s,\chi)\frac{\epsilon(\chi) L_u(s,\overline{\chi})}{s - (1 - s_0)} ds
\end{align*}
and
\begin{align*}
    I_{\mathscr{R}}(u, s_0, T, \lambda) &= \frac{1}{2 \pi i} \int_{\mathcal{C}_{\mathscr{R}}} g(s,\chi)\frac{L(s,\chi)}{s - s_0} ds + \frac{1}{2 \pi i} \int_{\mathcal{C}_{\mathscr{R}}} g(s,\chi)\frac{\epsilon(\chi) L(s,\overline{\chi})}{s - (1 - s_0)} ds \\
    &- \frac{1}{2 \pi i} \int_{\mathcal{C}_{\mathscr{R}}} g(s,\chi)\frac{L_u(s,\chi)}{s - s_0} ds - \frac{1}{2 \pi i} \int_{\mathcal{C}_{\mathscr{R}}} g(s,\chi)\frac{\epsilon(\chi) L_u(s,\overline{\chi})}{s - (1 - s_0)} ds.
\end{align*}

Letting $s = a + iT$, for $\mathcal{C}_{\mathscr{R}}$ we note by $(iv)$ in Lemma \eqref{lem:4} when $T \rightarrow \infty $, $ \left| \Gamma(a \pm iT) \right| \rightarrow 0.$\\

Now for $\mathcal{C}_{\mathscr{L}}$ note that as $T \rightarrow \infty$, by $(iii)$ in Lemma \eqref{lem:4} 
\begin{align*}
    \left| \Gamma \left( \frac{a + \kappa(\chi)}{2} + \frac{i T}{2} \right) \right| \sim C \left| \frac{a + \kappa(\chi)}{2} + \frac{i T}{2} \right|^{\frac{a + \kappa(\chi)}{2} - \frac{1}{2}} e^{-\left| \frac{a + \kappa(\chi)}{2} + \frac{i T}{2}\right| \frac{\pi}{2}}
\end{align*}
So since $a \in (-\lambda, \frac{1}{2})$ and $\kappa(\chi) = 0 \text{ or } 1$,  $\frac{a + \kappa(\chi)}{2} < 1$ so $ \frac{a + \kappa(\chi)}{2} - \frac{1}{2} < \frac{1}{2}$. Now we consider three cases:

$\textbf{Case 1:}$ If  $\frac{a + \kappa(\chi)}{2} = \frac{1}{2}$ then $$\left| \Gamma \left( \frac{a + \kappa(\chi)}{2} + \frac{i T}{2} \right) \right| \sim C \left| s \right|^{0} e^{-\left| \frac{a + \kappa(\chi)}{2} + \frac{i T}{2}\right| \frac{\pi}{2}}$$ which approaches $0$ as $T$ approaches $\infty$ because of the exponential decay. \\
$\textbf{Case 2:}$ If $\frac{a + \kappa(\chi)}{2} < \frac{1}{2}$ then $$\left| \Gamma \left( \frac{a + \kappa(\chi)}{2} + \frac{i T}{2} \right) \right| \sim C \left| s \right|^{-n} e^{-\left| \frac{a + \kappa(\chi)}{2} + \frac{i T}{2}\right| \frac{\pi}{2}}$$ which also approaches $0$ as $T$ approaches $\infty$ because of the exponential decay. \\
$\textbf{Case 3:}$ If $\frac{a + \kappa(\chi)}{2} > \frac{1}{2}$ then $$\left| \Gamma \left( \frac{a + \kappa(\chi)}{2} + \frac{i T}{2} \right) \right| \sim C \left| s \right|^{n} e^{-\left| \frac{a + \kappa(\chi)}{2} + \frac{i T}{2}\right| \frac{\pi}{2}}$$ with $n \in (0, \frac{1}{2})$ so we can use L'Hopital's rule to get $$C n \left| s \right|^{n - 1} \frac{4}{i \pi} e^{-\left| \frac{a + \kappa(\chi)}{2} + \frac{i T}{2}\right| \frac{\pi}{2}} $$ which, since $n - 1 < 0$, approaches $0$ as $T$ approaches $\infty$ because of the exponential decay. \\

Therefore for large $T$, the contribution of the horizontal segments is negligible when integrating over $\mathcal{C}_{\mathscr{L}}$ and $\mathcal{C}_{\mathscr{R}}$. So we find that:
\begin{align*}
    I_{\mathscr{R}}(u, s_0, T, \lambda) &= \frac{1}{2 \pi i} \int_{Re(s) = \sigma} g(s,\chi)\frac{L(s,\chi)}{s - s_0} ds + \frac{1}{2 \pi i} \int_{Re(s) = \sigma} g(s,\chi)\frac{\epsilon(\chi) L(s,\overline{\chi})}{s - (1 - s_0)} ds \\
    &- \frac{1}{2 \pi i} \int_{Re(s) = \sigma} g(s,\chi)\frac{L_u(s,\chi)}{s - s_0} ds - \frac{1}{2 \pi i} \int_{Re(s) = \sigma} g(s,\chi)\frac{\epsilon(\chi) L_u(s,\overline{\chi})}{s - (1 - s_0)} ds.
\end{align*}

Now note that $$\left| L_u(-\lambda + it, \chi) \right| = \prod_{p \le u} \left| \frac{1}{1 - \frac{\chi(p)}{p^{-\lambda + it}}} \right|.$$
We have 
\begin{align*}
\left| 1 - \chi(p)p^{-(-\lambda + it)} \right| &\ge \left| 1 - \left|\chi(p)p^{-(-\lambda + it)}\right| \right| = \left| 1 - \left|p^{-(-\lambda + it)}\right| \right| = \left| 1 - p^{\lambda} \right| \ge p^{\lambda} - 1 \sim p^{\lambda}.
\end{align*}

Therefore, $$\left| L_u(-\lambda + it, \chi) \right| \ll \prod_{p \le u} p^{-\lambda},$$ and hence as $T $ approaches $\infty$,

\begin{align*}
    I_{\mathscr{L}}(u, s_0, T, \lambda) \ll \left( \prod_{p \le u} p^{-\lambda} \right) \int_{-\infty}^{\infty} \left| g(-\lambda + it,\chi)\right| \left( \frac{1}{\left| -\lambda + it - s_0 \right|} + \frac{\left| \epsilon(\chi) \right|}{\left| -\lambda + it - (1 - s_0) \right|} \right) dt.
\end{align*}

Thus, as $\lambda, T \rightarrow \infty$, we have $I_{\mathscr{L}}(u, s_0, T, \lambda) \rightarrow 0$. \\

Recall 
\begin{align*}
\xi(s_0, \chi) - \xi^{\approx}_u(s_0,\chi) &= I(u, s_0, T, \lambda) - E(u, s_0, T, \lambda) \\
&= I_{\mathscr{L}}(u, s_0, T, \lambda) + I_{\mathscr{R}}(u, s_0, T, \lambda) - E(u, s_0, T, \lambda). 
\end{align*}
Since we found above that 
$$\lim_{T, \lambda \rightarrow \infty} I_{\mathscr{L}}(u, s_0, T, \lambda) = \lim_{T, \lambda \rightarrow \infty} E(u, s_0, T, \lambda) = 0,$$ we obtain the integral expression for the difference between the completed $L$-function and our approximation:
\begin{equation*}
    \xi(s_0, \chi) - \xi^{\approx}_u(s_0,\chi) = \frac{1}{2 \pi i} \int_{Re(s) = \sigma} g(s,\chi) \left( \frac{L(s,\chi) - L_u(s,\chi)}{s - s_0} + \epsilon(\chi) \frac{ L(s,\overline{\chi}) - L_u(s,\overline{\chi})}{s - (1 - s_0)} \right) ds.
\end{equation*}

\end{proof}

\subsection{Explicit formula for the error}

Fix a complex number $s_0$ throughout the section. In this section, we use the integral formula for $\xi_u(s_0, \chi) - \xi_u(s_0, \chi)$ from Theorem \ref{thm2} to prove the main result of Theorem \ref{thm1}. 

\begin{mydef}
        \label{def}
        The incomplete Gamma function $\Gamma (z,x)$ is a function defined for $\mathfrak{Re}(z)>0$ by
    \begin{equation*}
        \Gamma(z,x) = \int_x^\infty e^{-t} t^{z-1} dt.
    \end{equation*}
\end{mydef}

But the incomplete Gamma function can also be expressed as a series:
\begin{prop}
    For $a \neq 0, -1, -2, ...$
    \begin{equation*}
        \Gamma(a,x) = \Gamma(a) - \sum_{n = 0}^{\infty} \frac{(-1)^n x^{a + n}}{n! (a + n)}.
    \end{equation*}
\end{prop}
This comes from page 910 of \textit{Table of Integrals, Series, and Products} \cite{integrals}.

Defining $A_u$ to be the set of natural numbers with at least one prime factor greater than $u$, we find that
\begin{equation*}
    L(s,\chi) - L_u(s,\chi) = \sum_{n \in A_u} \frac{\chi(n)}{n^s}.
\end{equation*}

So therefore
\begin{align*}
    \xi(s_0, \chi) - \xi^{\approx}_u(s_0,\chi) &=  \frac{1}{2 \pi i} \int_{Re(s) = \sigma} g(s,\chi) \left( \frac{\sum_{n \in A_u} \frac{\chi(n)}{n^s}}{s - s_0} + \epsilon(\chi) \frac{ \sum_{n \in A_u} \frac{\overline{\chi}(n)}{n^s} }{s - (1 - s_0)} \right) ds \\
    &= \sum_{n \in A_u} \frac{1}{2 \pi i} \int_{Re(s) = \sigma} g(s,\chi) \frac{\chi(n)}{n^s(s - s_0)} ds \\
    &+ \sum_{n \in A_u} \frac{1}{2 \pi i} \int_{Re(s) = \sigma} g(s,\chi) \frac{\epsilon(\chi) \overline{\chi}(n)}{n^s(s - 1 + s_0)} ds.
\end{align*}
Define the notation:

\begin{equation}
    J_1(n)= \frac{1}{2 \pi i} \int_{Re(s) = \sigma} g(s,\chi) \frac{\chi(n)}{n^s(s - s_0)} ds
\end{equation}
\begin{equation}
    J_2(n) = \frac{1}{2 \pi i} \int_{Re(s) = \sigma} g(s,\chi) \frac{\epsilon(\chi) \overline{\chi}(n)}{n^s(s - 1 + s_0)} ds.
\end{equation}
Now we work towards evaluating $J_1$. As with before, $s = \lambda + it$.

\begin{align*}
    J_1(n) &= \frac{1}{2 \pi i} \int_{Re(s) = \sigma} g(s,\chi) \frac{\chi(n)}{n^s(s - s_0)} ds \\
    &= \sum_{\rho \text{ zeros of } g(s , \chi)} \text{Res}(\rho) + g(s_0, \chi) \chi(n)n^{-s_0}\\
    &= 2 \chi(n) n^{\kappa(\chi)} \sum_{h = 0}^{\infty} \left( \frac{(-1)^h \pi^h n^{2h}}{q^h \Gamma(h + 1)} \frac{1}{-2h - \kappa(\chi) - s_0} \right) + g(s_0, \chi) \chi(n)n^{-s_0}\\
    &= - \chi(n) n^{\kappa(\chi)} \sum_{h = 0}^{\infty} \left( \frac{(-1)^h \pi^h n^{2h}}{q^h h!} \frac{1}{h + \frac{\kappa(\chi) + s_0}{2}} \right) + g(s_0, \chi) \chi(n)n^{-s_0}\\
    &= - \chi(n) n^{\kappa(\chi)} \left( \frac{q}{n^2 \pi} \right)^{\frac{\kappa(\chi) + s_0}{2}} \sum_{h = 0}^{\infty} \left( \frac{(-1)^h}{h!} \left( \frac{n^2 \pi}{q} \right)^{h + \frac{\kappa(\chi) + s_0}{2}}  \frac{1}{h + \frac{\kappa(\chi) + s_0}{2}} \right) + g(s_0, \chi) \chi(n)n^{-s_0}\\
    &= - \chi(n) n^{\kappa(\chi)} \left( \frac{q}{n^2 \pi} \right)^{\frac{\kappa(\chi) + s_0}{2}} \left( \Gamma \left( \frac{\kappa(\chi) + s_0}{2} \right) - \Gamma\left( \frac{\kappa(\chi) + s_0}{2}, \frac{n^2 \pi}{q} \right) \right) + g(s_0, \chi) \chi(n)n^{-s_0}.
\end{align*}
Note that $$n^{\kappa(\chi)} \cdot n^{-2 \left( \frac{\kappa(\chi) + s_0}{2} \right)} = n^{-s_0},$$ therefore $$g(s_0, \chi) \chi(n) n^{-s_0} = \left( \frac{q}{n^2 \pi} \right)^{\frac{\kappa(\chi) + s_0}{2}} \Gamma \left( \frac{\kappa(\chi) + s_0}{2} \right) \chi(n) n^{\kappa(\chi)}.$$ So we can utilize this to simplify $J_1$ to get
\begin{equation}
    J_1(n) = \chi(n) n^{\kappa(\chi)} \left( \frac{q}{n^2 \pi} \right)^{\frac{\kappa(\chi) + s_0}{2}}  \Gamma\left( \frac{\kappa(\chi) + s_0}{2}, \frac{n^2 \pi}{q} \right).
\end{equation}
Similarly, we can simplify $J_2$ to get
\begin{equation}
    J_2(n) = \epsilon(\chi) \overline{\chi}(n) n^{\kappa(\chi)} \left( \frac{q}{n^2 \pi} \right)^{\frac{\kappa(\chi) + 1 - s_0}{2}}  \Gamma\left( \frac{\kappa(\chi) + 1 - s_0}{2}, \frac{n^2 \pi}{q} \right).
\end{equation}

Therefore
\begin{align*}
    J_1(n) + J_2(n) &= \chi(n) n^{\kappa(\chi)} \left( \frac{q}{n^2 \pi} \right)^{\frac{\kappa(\chi) + s_0}{2}}  \Gamma\left( \frac{\kappa(\chi) + s_0}{2}, \frac{n^2 \pi}{q} \right) \\
    &+ \epsilon(\chi) \overline{\chi}(n) n^{\kappa(\chi)} \left( \frac{q}{n^2 \pi} \right)^{\frac{\kappa(\chi) + 1 - s_0}{2}}  \Gamma\left( \frac{\kappa(\chi) + 1 - s_0}{2}, \frac{n^2 \pi}{q} \right).
\end{align*}
This concludes the proof of Theorem \ref{thm1}.

\subsection{Bounding the error}

Fix a complex number $s_0$ throughout the section. In this section, we use the explicit formula for the error to bound the error term.

From Mathematica, we can expand the Gamma factor in $J_1(n)$:
\begin{align}
    \Gamma\left( \frac{\kappa(\chi) + s_0}{2}, \frac{n^2 \pi}{q} \right) &= \exp \left( \frac{- \pi n^2}{q} + O \left[\frac{1}{n} \right]^2 \right) \label{bd1} \\
    & \times n^{\kappa(\chi) + s_0 - 2} \left(  \left( \frac{\pi}{q} \right)^{\frac{\kappa(\chi) + s_0}{2} - 1} + O \left[\frac{1}{n} \right]^2 \right). \label{bd2}
\end{align}
It is apparent that \ref{bd1} decays exponentially while \ref{bd2} grows in a polynomial rate, therefore the Gamma factor decays exponentially.
Similarly for the Gamma term in $J_2(n)$,
\begin{align*}
    \Gamma\left( \frac{\kappa(\chi) + 1 - s_0}{2}, \frac{n^2 \pi}{q} \right) &= \exp \left( \frac{- \pi n^2}{q} + O \left[\frac{1}{n} \right]^2 \right) \\
    & \times n^{\kappa(\chi) - 1 - s_0} \left( \left( \frac{\pi}{q} \right)^{\frac{\kappa(\chi) + 1 - s_0}{2} - 1} + O \left[\frac{1}{n} \right]^2 \right).
\end{align*}

Summing on $A_u$ is less than summing over all integers greater than $u$, so we can obtain an upper bound by summing
\begin{equation*}
    |\xi(s_0, \chi) - \xi^{\approx}_u(s_0,\chi)| \le \sum_{n=u}^\infty |J_1(n)| + |J_2(n)|.
\end{equation*}
This sum itself is bounded by integrating
\begin{equation*}
    |\xi(s_0, \chi) - \xi^{\approx}_u(s_0,\chi)| \le \int_{u}^\infty |J_1(x)|dx + \int_{u}^\infty |J_2(x)| dx.
\end{equation*}
Using the main terms, from our previous expansions of the Gamma factors, along with the fact that $|\chi(n)| \le 1$ always and $|\epsilon(\chi)| = 1$ for primitive characters,
\begin{align*}
    |\xi(s_0, \chi) - \xi^{\approx}_u(s_0,\chi)| &\ll \int_{u}^\infty x^{\kappa(\chi)} \left(\frac{q}{\pi}\right)^{\frac{\kappa(\chi)+s_0}{2}} x^{-\kappa(\chi)-s_0} \exp \left( \frac{- \pi x^2}{q} \right) x^{\kappa(\chi) + s_0 - 2} \left( \frac{\pi}{q} \right)^{\frac{\kappa(\chi) + s_0}{2} - 1} dx \\
    &+ \int_{u}^\infty x^{\kappa(\chi)} \left(\frac{q}{\pi}\right)^{\frac{\kappa(\chi)+1-s_0}{2}} x^{-\kappa(\chi)-1+s_0} \exp \left( \frac{- \pi x^2}{q} \right) x^{\kappa(\chi) - 1 - s_0} \left( \frac{\pi}{q} \right)^{\frac{\kappa(\chi) + 1 - s_0}{2} - 1} dx.
\end{align*}
This simplifies to
\begin{equation*}
    |\xi(s_0, \chi) - \xi^{\approx}_u(s_0,\chi)| \ll \int_{u}^\infty 2 \left(\frac{q}{\pi}\right) x^{\kappa(\chi) - 2} \exp \left( \frac{- \pi x^2}{q} \right) dx.
\end{equation*}
Computing this integral, we find that the error term is bounded by
\begin{equation*}
    |\xi(s_0, \chi) - \xi^{\approx}_u(s_0,\chi)| \ll \begin{cases}
        \left(\frac{q}{\pi}\right) \Gamma \left(0, \frac{\pi u^2}{q}\right) & \text{if } \kappa(\chi) = 1\\
        \sqrt{\frac{q}{\pi}}\Gamma\left( -\frac{1}{2}, \frac{\pi u^2}{q} \right) & \text{if } \kappa(\chi) = 0.
    \end{cases}
\end{equation*}
This concludes the proof of Corollary \ref{cor}.

\section{Acknowledgements}
The study resulting in this paper was assisted by a grant from the Baker Program in Undergraduate Research, which is administered by Northwestern University's Weinberg College of Arts and Sciences. However, the conclusions, opinions, and other statements in this paper are the author's and not necessarily those of the sponsoring institution.

I want to thank Professor Nastasescu for mentoring me and helping me throughout the entire process resulting in this paper. I could not be where I am without her guidance and mentorship.

\printbibliography[
heading=bibintoc,
title={Bibliography}
]

@book{davenport,
    author = {H. ~Davenport},
    year = {1980},
    title = {Multiplicative Number Theory},
    publisher = {Princeton University Press}
}

@book{integrals,
    author = {D. ~Zwillinger, I.S. ~Gradshteyn, and I.M. ~Ryzhik},
    year = {2015},
    title = {Table of Integrals, Series, and Products (Eighth Edition)},
    publisher = {Academic Press}
}

@article{nastasescu,
    author={M. ~Nastasescu and A. ~Zaharescu},
    title={A class of approximations to the Riemann zeta function},
    journal={Journal of Mathematical Analysis and Applications},
    volume={514},
    number={2},
    year={2022},
    pages={}
}

@article{matiy1,
    author={Y. ~Matiyasevich},
    title={A Few Factors from the Euler Product Are Sufficient for Calculating the Zeta Function with High Precision},
    journal={Proceedings of the Steklov Institute of Mathematics},
    volume={299},
    number={},
    year={2017},
    pages={178 -- 188}
}

@article{matiy2,
    author = {Y. ~Matiyasevich},
    title = {Riemann's zeta function and finite Dirichlet series},
    journal = {St. Petersburg Mathematical Journal},
    volume = {27},
    number={2},
    year = {2016},
    pages = {985 -- 1002}
}

@ARTICLE{Gonek1,
	author = {S.M. ~Gonek},
	title = {Finite Euler Products and the Riemann Hypothesis},
	journal = {Transactions of the American Mathematical Society},
	volume = {364},
	number = {4},
	year = {2012},
	pages = {2157 -- 2191}
}

@article{Gonek2,
    author = {S.M. ~Gonek, C.P. ~Hughes, and J.P. ~Keating},
    title = {{A hybrid Euler-Hadamard product for the Riemann zeta function}},
    journal = {Duke Mathematical Journal},
    volume = {136},
    number = {3},
    year = {2007},
    pages = {507 -- 549}
}

@article{Gonek3,
    author = {S. M. ~Gonek and J. P. ~Keating},
    title = {Mean values of finite Euler products},
    journal = {Journal of the London Mathematical Society},
    volume = {82},
    number = {3},
    year = {2010},
    pages = {763 -- 786}
}

@misc{lmfdb,
  author       = {The {LMFDB Collaboration}},
  title        = {The {L}-functions and modular forms database},
  howpublished = {\url{https://www.lmfdb.org}},
  year         = {2023},
  note         = {[Online; accessed 20 April 2023]},
}

\end{document}